\documentclass[a4paper, 10pt]{amsart}
\usepackage[utf8]{inputenc}
\usepackage[T1]{fontenc}
\usepackage{lmodern}
\usepackage[english]{babel}
\usepackage[dvipsnames]{xcolor}
\usepackage{adjustbox}
\usepackage{letltxmacro}
\usepackage{todonotes}
\LetLtxMacro\todonotestodo\todo
\renewcommand{\todo}[2][]{\todonotestodo[#1]{TODO: {#2}}}
\usepackage{enumerate}
\usepackage{hyperref}
\usepackage{url}

\makeatletter
\newtheorem*{rep@theorem}{\rep@title}
\newcommand{\newreptheorem}[2]{%
\newenvironment{rep#1}[1]{%
 \def\rep@title{#2 \ref{##1}}%
 \begin{rep@theorem}}%
 {\end{rep@theorem}}}
\makeatother

\newreptheorem{theorem}{Theorem}

\newtheorem{lemma}{Lemma}[section]

\theoremstyle{definition}

\newtheorem*{claim*}{Claim}

\newtheorem*{theorem*}{Theorem}
\newtheorem*{corollary*}{Corollary}
\newtheorem*{lemma*}{Lemma}
\newtheorem{counterexample}[lemma]{Counterexample}

\newcommand{\Z}{\mathbb{Z}}
\newcommand{\N}{\mathbb{N}}

\DeclareMathOperator{\im}{im}
\DeclareMathOperator{\Pic}{Pic}

\title{On the Distribution of Prime Divisors in Krull Monoid Algebras}
\author{Victor Fadinger}
\address[Victor Fadinger]{Institut für Mathematik und Wissenschaftliches Rechnen\\Karl-Franzens-Universität Graz\\
  Heinrichstraße 36\\8010 Graz\\Austria}
\email{\href{mailto:victor.fadinger@uni-graz.at}{victor.fadinger@uni-graz.at}}
\thanks{V.~Fadinger is supported by the Austrian Science Fund (FWF): W1230}

\author{Daniel Windisch}
\address[Daniel Windisch]{Institut für Analysis und Zahlentheorie\\Technische Universität Graz\\
  Kopernikusgasse 24/II\\8010 Graz\\Austria}
\email{\href{mailto:dwindisch@math.tugraz.at}{dwindisch@math.tugraz.at}}
\thanks{D.~Windisch is supported by the Austrian Science Fund (FWF): P~30934}

\begin{document}

\begin{abstract}
In the present work, we prove that every class of the divisor class group of a Krull monoid algebra contains infinitely many prime divisors. Several attempts to this result have been made in the literature so far, unfortunately with open gaps. We present a complete proof of this fact.
\end{abstract}

\maketitle

\section{Introduction}

The investigation of class groups and the distribution of prime divisors in the classes is a central topic in ring theory and has been studied for many classes of rings, e.g. for orders in number fields. A variety of realization theorems for class groups have been achieved, for instance that every abelian group occurs as the class group of a Dedekind domain (Claborn's Realization Theorem) respectively as the class group of a simple Dedekind domain \cite{Sm}. \\
In the present paper, we study the distribution of prime divisors in the class group of commutative monoid algebras $D[S]$ that are Krull. In our setting, a monoid $S$ is a cancellative commutative unitary semigroup. For a commutative ring $D$ with unity and a monoid $S$, we denote by $D[S]$ the monoid algebra of $D$ over~$S$. It is a well-known result by Chouinard, that $D[S]$ is a Krull domain if and only if $D$ is a Krull domain and $S$ is a torsion-free Krull monoid with $S^\times$ satisfying the ACC (ascending chain condition) on cyclic subgroups, where $S^\times$ denotes the group of units of $S$ (see \cite{Chou81}). Moreover, we have a canonical isomorphism of divisor class groups $\mathcal{C}_v(D) \oplus \mathcal{C}_v(S) \cong \mathcal{C}_v(D[S])$ via $([I],[J]) \mapsto [I[J]]$ \cite[Corollary 16.8]{Gil84}. \\
In the setting of monoids, it is known that every abelian group is the class group of a Krull monoid and the possible ways of distributing the prime divisors in the classes are completely determined \cite[Theorem 2.5.4]{GHK}. Considering domains, a slightly weaker form of a realization result holds true. Although every abelian group is the class group of a Dedekind domain, it is not completely clear how prime divisors behave \cite[Theorem 3.7.8]{GHK}.\\
In Section \ref{section:preliminaries}, we fix notations concerning monoid algebras and class groups.
Section \ref{section:proof} contains a proof of the following result concerning the distribution of prime divisors:

\begin{theorem*}
Let $D$ be a Krull domain and let $S$ be a Krull monoid such that $D[S]$ is a Krull domain. Then every divisor class of $D[S]$ contains infinitely many prime divisors.
\end{theorem*}

In the special case of the polynomial ring over a Krull domain, this is already known \cite[Theorem 45.5]{Gil72}. Moreover, Kainrath \cite{Kain} proved the same for finitely generated algebras over infinite fields. \\
Attempts to achieve such results for Krull monoid algebras have been made by Kim \cite{Kim} and Chang \cite{Chang}. However, there are deficiencies in their proofs as we demonstrate in Section \ref{section:examples}.

\section{Preliminaries} \label{section:preliminaries}
We assume some familiarity with monoid algebras and ideal systems. For a domain $D$ and a monoid $S$, let $D[S]$ denote the monoid algebra of $D$ over $S$. It is well known that the monoid algebra $D[S]$ is a domain if and only if the monoid $S$ is torsion-free \cite[Theorem 8.1]{Gil84} and in that case $S$ admits a total order $<$ compatible with its semigroup operation. Therefore, we can write every element $f\in D[S]$ in the form $f=\sum_{i=1}^{n}d_iX^{s_i}$ with $d_i\in D$, $s_i\in S$ and $s_1<s_2<\hdots <s_n$. For subsets $P\subseteq D$ and $H\subseteq S$, we denote by $P[H]=\{\sum_{i=1}^{n}p_iX^{h_i}\mid n\in\N, p_i\in P, h_i\in H\}\subseteq D[S]$.\\
Let $K$ be the quotient field of $D$ and let $G$ be the quotient group of $S$. For an element $f = \sum_{i=1}^n d_i X^{s_i} \in K[G]$ with $d_i \in K$ and $s_i \in G$ we denote by $A_f$ resp. $E_f$ the fractional ideal of $D$ resp. $S$ generated by $d_1,\ldots,d_n$ resp. $s_1,\ldots,s_n$.\\
Let $G$ be an abelian group. An element $g \in G$ is said to be of \textit{height} $(0,0,0,\ldots)$ if the equation $px = g$ has no solution $x \in G$ for any prime number $p$. Furthermore, $g$ is said to be of \textit{type} $(0,0,0,\ldots)$ if the same equation has a solution for only finitely many prime numbers. The group $G$ is said to be of \textit{type} $(0,0,0,\ldots)$ if every element of $G$ is of type $(0,0,0,\ldots)$.
It is well known that $G$ is of type $(0,0,0,\ldots )$ if and only if it satisfies the ACC (ascending chain condition) on cyclic subgroups. For a proof of this statement and for further equivalent conditions, see \cite[\S 14]{Gil84}.\\
Let $S$ be a monoid. By $\mathcal{C}_t(S)$ we denote the the $t$-class group of $S$ which is the quotient group of the group of $t$-invertible fractional $t$-ideals of $S$ modulo the subgroup of non-empty principal fractional ideals of $S$. If $S$ is a Krull monoid, then $\mathcal{C}_t(S)$ coincides with the $v$-class group $\mathcal{C}_v(S)$ (which is isomorphic to the divisor class group of $S$). Moreover, in this case $\mathcal{C}_v(S)$ is isomorphic to the quotient of the monoid of $v$-invertible (integral) $v$-ideals modulo the submonoid of non-empty principal (integral) ideals. \\
If $I$ is a $t$-invertible fractional $t$-ideal of $S$, we denote by $[I] \in \mathcal{C}_t(S)$ its $t$-class and refer to it as the divisor class of $I$ if $S$ is a Krull monoid. Moreover, we denote by $\mathfrak{X}(S)$ the set of height-one prime ideals of $S$. In the literature, the elements of $\mathfrak X(S)$ are also often called prime divisors of $S$ (especially, if $S$ is a Krull monoid).\\
Replacing the monoid $S$ by a domain $D$, the terminology and notation concerning class groups and prime divisors above is the same. For further information, see \cite{GHK}.\\
We will make intensive use of the theory of valuations on fields and groups. If $D$ is a domain and $P$ is a prime ideal of $D$ such that the localization $D_P$ is a valuation domain, then we denote by $\mathsf{v}_P$ its valuation induced on the quotient field of $D$ and call it the $P$-adic valuation. In the case of monoids, the notation is analogous. The interested reader is referred to \cite{Bourbaki} and \cite{Gil84}. \\
For the remainder of this work and if not specified otherwise, $D$ denotes a domain with quotient field $K$ and $S$ is a torsion-free monoid with quotient group $G$.

\section{Proof of the main result} \label{section:proof}

To ensure the existence of prime divisor in divisor classes of $D[S]$, it is necessary to construct irreducible elements in $K[G]$. The following lemma (which is due to Matsuda \cite{Mat3}) is a first step into this direction that will be used in a special case. For the convenience of the reader we include the full proof.

\begin{lemma} \emph{(}\cite[Lemma 4.1]{Mat3}, \cite[Lemma 2.2]{Mat1}\emph{)} \label{lemma:Matsuda}
Let $G$ be a torsion-free abelian group and let $K$ be a field. Let $g \in G$ be an element of height $(0,0,0,\ldots)$ and $a,b \in K \setminus \{0\}$. Then $a+bX^g$ is irreducible in $K[G]$.
\end{lemma}

\begin{proof}
Let $f,h \in K[G]$ such that $a+bX^g = fh$. Let $H$ be the subgroup of $G$ that is generated by $g$ together with the exponents of $f$ and $h$. 
We first show, that the subgroup $\langle g \rangle$  of $H$ generated by $g$ is a direct summand of $H$. Let $p$ be a prime number. Suppose $p^nx \in \langle g \rangle$ for some positive integer $n$ and $x \in H$. Then there exists $m \in \Z$ such that $p^nx = mg$. Since $H$ is torsion-free, there are integers $n'\geq0$ and $m'$ such that $p^{n'}x = m'g$ and $\gcd(p,m') = 1$. It follows that $1 = kp^{n'} + l m'$ for some $k,l \in \Z$. Multiplying by $g$ gives $g = kp^{n'}g + lm'g = p^{n'} (kg+lx)$. Since the height of $g$ equals $(0,0,0,\ldots)$, we have $n' = 0$, hence $x = m'g \in \langle g \rangle$. So we have shown that $\langle g\rangle$ is a pure subgroup of the finitely generated group $H$ and is therefore a direct summand by \cite[Corollary 25.3]{Fuchs}. So we write $H = \langle g \rangle \oplus \langle e_1 \rangle \oplus \ldots \oplus \langle e_n \rangle$ for some $e_1, \ldots, e_n \in H$. \\
The set $\{X^g,X^{e_1}, \ldots,X^{e_n}\}$ is algebraically independent over $K$. Hence $a+bX^g$ is irreducible in the polynomial ring $K[X^g,X^{e_1},\ldots, X^{e_n}]$ and therefore also in $K[H] = K[X^g,X^{-g},X^{e_1},X^{-e_1}, \ldots, X^{e_n},X^{-e_n}]$. It follows that either $f$ or $h$ is a unit in $K[H]$ and hence in $K[G]$.
\end{proof}

The next result is a special case of \cite[Theorem 3]{Chang}. Nevertheless, we recall its proof.

\begin{lemma} \cite[Theorem 3]{Chang} \label{lemma:Chang}
Let $D[S]$ be a Krull monoid algebra and $f \in K[G] \setminus \{0\}$. Then $fK[G] \cap D[S] = fA_f^{-1}[E_f^{-1}]$.
\end{lemma}

\begin{proof}
As noted in the introduction, we have an isomorphism $\mathcal{C}_v(D) \oplus \mathcal{C}_v(S) \cong\mathcal{C}_v(D[S])  $ via $ ([I],[J]) \mapsto [I[J]]$. Clearly, $fK[G] \cap D[S]$ is a height-one prime ideal of $D[S]$ and hence non-zero divisorial. So there exist divisorial ideals $I \neq (0)$ resp. $J \neq \emptyset$ of $D$ resp. $S$ and $h \in K[G] \setminus \{0\}$ such that $fK[G] \cap D[S] = hI[J]$ (note that $fK[G] \cap D[S] \subseteq K[G]$). It follows that $fK[G] = hK[G]$ and hence $h = uX^af$ for some $u \in K\setminus \{0\}$ and $a \in G$. Thus, $fK[G] \cap D[S] = fuI[a + J]$. A simple computation gives $uI = A_f^{-1}$ and $a+ J = E_f^{-1}$.
\end{proof}

\begin{lemma}\label{lemma:approximation}
Let $D$ be a Krull domain with quotient field $K \neq D$ and let $I$ be a non-zero divisorial ideal of $D$. There exist $a,b \in K$ and $P \in \mathfrak{X}(D)$ such that $I^{-1} = (a,b)_v$ and $\mathsf{v}_P(\frac{a}{b}) = 1$. If $D$ is not semi-local, then there are infinitely many $P \in \mathfrak{X}(D)$ such that there exist $a,b \in K$ with $I^{-1} = (a,b)_v$ and $\mathsf{v}_P(\frac{a}{b}) = 1$.
\end{lemma}

\begin{proof}
By the Approximation Theorem for Krull domains, there exist $a',b \in K$ such that $I^{-1} = (a',b)_v$. If $I^{-1}$ is a principal fractional ideal, $I$ is a principal ideal, hence we can assume without loss of generality that $b$ is its generator. Now take any $P \in \mathfrak{X}(D)$ and any element $a \in D$ with $\mathsf{v}_P(a) = 1$. In particular, if $D$ is not semi-local, there exist infinitely many such $P$. \\
Now assume that $I^{-1}$ is not a principal fractional ideal. Thus, there exists $P \in \mathfrak{X}(D)$ such that $\mathsf{v}_P(a') \neq \mathsf{v}_P(b)$. Without loss of generality, let $\mathsf{v}_P(a') > \mathsf{v}_P(b)$. By the Approximation Theorem for Krull domains, we may choose $a \in K$ such that
\begin{align*}
\mathsf{v}_Q(a) = \begin{cases} \mathsf v_P(b) + 1 &\mbox{if } Q = P, \\
\mathsf v_Q(a') & \mbox{if } \mathsf v_Q(a') \neq 0 \text{ or } \mathsf v_Q(b) \neq 0 \text{ and } Q \neq P, \\
\mathsf{v}_Q(a) \geq 0  & \mbox{otherwise}, \end{cases}
\end{align*}
for $Q \in \mathfrak{X}(D)$. Then $\min \{ \mathsf v_Q(a),\mathsf v_Q(b) \} = \min \{\mathsf{v}_Q(a'), \mathsf{v}_Q(b)\}$ for all $Q \in \mathfrak{X}(D)$ and hence $I^{-1} = (a,b)_v$. Moreover $\mathsf{v}_P(\frac{a}{b}) =~1$.\\
If $D$ is not semi-local, then $\mathfrak{X}(D)$ is infinite, whence there are infinitely many $P \in \mathfrak{X}(D)$ such that $\mathsf{v}_P(a') = 0 = \mathsf{v}_P(b)$. For each of them, we construct $a$ in the following way using the Approximation Theorem for Krull domains: 
\begin{align*}
\mathsf{v}_Q(a) = \begin{cases} 1 &\mbox{if } Q = P, \\
\mathsf v_Q(a') & \mbox{if } \mathsf v_Q(a') \neq 0 \text{ or } \mathsf v_Q(b) \neq 0 \text{ and } Q \neq P, \\
\mathsf{v}_Q(a) \geq 0  & \mbox{otherwise}, \end{cases}
\end{align*}
for $Q \in \mathfrak{X}(D)$. Then again $\min \{ \mathsf v_Q(a),\mathsf v_Q(b) \} = \min \{\mathsf{v}_Q(a'), \mathsf{v}_Q(b)\}$ for all $Q \in \mathfrak{X}(D)$ and hence $I^{-1} = (a,b)_v$. Moreover $\mathsf{v}_P(\frac{a}{b}) =~1$.
\end{proof}

Kim \cite{Kim} showed that every divisor class of $D[G]$ (where $G \neq \{0\}$)  contains a prime divisor. We copy and modify his proof in such a way that the existence of infinitely many prime divisors in each class follows.

\begin{lemma} \label{lemma:Kim}
Let $G$ be a non-zero abelian group such that $D[G]$ is a Krull domain. Then each divisor class of $D[G]$ contains infinitely many prime divisors.
\end{lemma}

\begin{proof}
Note that $G$ is torsion-free, because $D[G]$ is Krull. \\
\textbf{Case 1:} Let $G$ be finitely generated. Then there is a positive integer $n$ such that $D[G]$ is isomorphic to the Laurent polynomial ring in $n$ indeterminates over $D$. If $\mathfrak{X}(D)$ is finite, then $D$ is a semi-local principal ideal domain. Hence $D[G]$ is factorial with infinitely many non-associated prime elements. Now let $\mathfrak{X}(D)$ be infinite. Let $I[G]$ be a non-zero divisorial ideal of $D[G]$ where $I$ is a non-zero divisorial ideal of $D$. By Lemma \ref{lemma:approximation}, there exist infinitely many $P \in \mathfrak{X}(D)$ such that there exist $a,b \in K$ with $I^{-1} =(a,b)_v$ and $\mathsf{v}_P(\frac{a}{b}) = 1$. For each choice of $P$ and corresponding $a,b \in K$ with the above properties, we construct a prime divisor lying in $[I[G]]$: Let $\alpha \in G$ with $\alpha > 0$ and set $g = \frac{a}{b} + X^\alpha$. Then $g$ is irreducible in $D_P[G]$ by a generalized version of Eisenstein's criterion \cite[Lemma 5]{Chang}. Hence $g$ is irreducible in $K[G]$ by localization. Moreover, we have by Lemma \ref{lemma:Chang} that $gK[G] \cap D[G] = g (\frac{a}{b},1)^{-1}[G] = g b (a,b)^{-1} [G] = gb I[G]$ and hence $gK[G] \cap D[G]$ is one of infinitely many prime divisors lying in the class of $I[G]$. \\
\textbf{Case 2:} Let $G$ be non-finitely generated. Again, let $I[G]$ be a non-zero divisorial ideal of $D[G]$ where $I$ is a non-zero divisorial ideal of $D$. By Lemma \ref{lemma:approximation}, there exist $a,b \in K$ such that $I^{-1} = (a,b)_v$. Since $G$ is non-finitely generated, there are infinitely many elements $g \in G$ of height $(0,0,0,\ldots)$. For each choice of $g$, the element $f = a + b X^g$ is irreducible in $ K[G]$ by Lemma \ref{lemma:Matsuda}. Moreover, by Lemma \ref{lemma:Chang} we have $fK[G] \cap D[G] = f(a,b)^{-1}[G] = fI[G]$ and hence $fK[G] \cap D[G]$ is one of infinitely many prime divisors lying in the class of $I[G]$.
\end{proof}

\begin{lemma} \label{lemma:field}
Let $K$ be a field and let $S$ be a monoid such that $K[S]$ is a Krull domain. Then every divisor class of $K[S]$ contains infinitely many prime divisor.
\end{lemma}

\begin{proof}
If $S$ is a Krull monoid, then it has a decomposition $S\cong S_{red}\times S^\times$, where $S_{red}$ denotes the reduced monoid of $S$ \cite[Theorem 2.4.8.2]{GHK}. Thus, by \cite[Theorem 7.1]{Gil84} we have $K[S] \cong K[S_{red}][S^\times]$, which is a Krull group algebra over $K[S_{red}]$. Therefore, if $S$ is non-reduced then $S^\times$ is non-trivial and we can use Lemma \ref{lemma:Kim} to show that every divisor class contains infinitely many prime divisors. So, from now on, assume that $S$ is reduced.\\
Let $G$ be the quotient group of $S$.

\textbf{Case 1:} First assume that $S$ is finitely generated. Then, since $S$ is torsion-free, it is isomorphic to an additive submonoid of the group $(\Z^m,+)$ for some $m \in \N$ \cite[page 50]{Bruns}. Therefore $G$ is a $\Z$-submodule of the free $\Z$-module $\Z^m$ and hence free, say of rank $n \in \N$.

\smallskip
\noindent
\textbf{Claim A:} There exists a $\Z$-basis $B$ of $G$ such that $B \cap S \neq \emptyset$.
\begin{proof}[Proof of Claim A]
If $n = 0$ or $n = 1$, this is trivial. So let $n >1$. Then in particular $S \neq G$, because $S$ is reduced. Let $P\subseteq S$ be a height-one prime ideal of $S$ and $\mathsf{v}_P: G \to \Z$ the associated $P$-adic valuation. Let $a \in S$ with $\mathsf{v}_P(a) = 1$. Let $G_0$ be the kernel of $\mathsf{v}_P$. Then $G_0$ is a free $\Z$-module of rank $n-1$ and we have $G \cong \ker \mathsf{v}_P \oplus \im \mathsf{v}_P \cong G_0 \oplus \Z$. So, if $a_1,\ldots,a_{n-1}$ is a $\Z$-basis of $G_0$, then $a_1,\ldots,a_{n-1},a$ is a $\Z$-basis of $G$ with $a \in S$.
\qedhere[Proof of Claim A]
\end{proof}
Now let $(a_1,\ldots,a_n)$ be a $\Z$-basis of $G$ with $a_1 \in S$. We define $\mathcal{O} = K[a_1]$ with quotient field $F = K(a_1)$. By $L = K(a_1,\ldots,a_n)$ we denote the quotient field of $K[S]$. Then $K[S]$ is a finitely generated $\mathcal{O}$-algebra, $F$ is a Hilbertian field (as a finitely generated transcendental extension of a field), the Krull dimension of $\mathcal{O}$ equals $1$, $\mathcal{O}$ and $K[S]$ are integrally closed, $L/F$ is a purely transcendental extension and hence separabel and regular, and $\Pic(\mathcal{O}) = 0$. Therefore by \cite[Theorem 2]{Kain}, $K[S]$ has infinitely many prime divisors in all classes and we are done.

\textbf{Case 2:} Now let $S$ be non-finitely generated. We have an isomorphism of divisor class groups $\mathcal{C}_v(S) \to \mathcal{C}_v(K[S])$ via $[I] \mapsto [K[I]]$ (\cite[Theorem 16.7]{Gil84}). Therefore it suffices to prove that $[K[I]]$ contains a prime ideal of $K[S]$ for every non-empty $v$-ideal $I$ of $S$. So let $I$ be a non-empty $v$-ideal of $S$ and let $g_1,\ldots,g_n \in G$ with $I^{-1} = ((g_1+S) \cup \ldots \cup (g_n + S))_v$. \\
Since $S$ is a reduced non-finitely generated Krull monoid, it follows from \cite[Theorem 2.7.14]{GHK} that the set $\mathfrak X(S)$ of height-one prime ideals of $S$ is infinite. Thus, there exist infinitely many $P \in \mathfrak{X}(S)$ such that $0 = \mathsf{v}_P(g_1) = \ldots = \mathsf{v}_P(g_n)$. For each choice of $P$, we construct a prime divisor lying in $[K[I]]$ (and infinitely many of them are pairwise distinct). Let $a \in S$ with $\mathsf{v}_P(a) = 1$. Let $g = X^{g_1} + \ldots + X^{g_n} + X^{g_n+a} \in K[G]$.

\smallskip
\noindent
\textbf{Claim B:} $g$ is irreducible in $K[G]$. \\
Suppose that the claim holds true. On the one hand, $I^{-1} = ((g_1+S) \cup \ldots \cup (g_n+S))_v = ((g_1+S) \cup \ldots \cup (g_n+S) \cup (g_n + a +S))_v$, because $a \in S$. On the other hand, by Lemma \ref{lemma:Chang}, we have that $gK[G] \cap K[S] = g K[((g_1+S) \cup \ldots \cup (g_n+S) \cup (g_n + a +S))^{-1}] = gK[I_v] = gK[I]$. Therefore, $gK[G] \cap K[S]$ is one of infinitely many prime ideals of $K[S]$ (because $g$ is irreducible by the claim) that lies in the divisor class of $K[I]$.

\noindent
\begin{proof}[Proof of Claim B]
It suffices to show that $g$ is irreducible in $K[S_P]$, because $K[G]$ is the localization of $K[S_P]$ at the set $\{X^s \mid s \in P\}$ and $K[S_P]$ is a factorial domain.\\
To see that $g \in K[S_P]$ is irreducible, let $h_1, h_2 \in K[S_P]$ with $g = h_1h_2$. Since $S_P$ is a discrete rank one valuation monoid, the map $\N_0 \times S_P^\times \to S_P$ via $ (n,s) \mapsto sa^n$ is an isomorphism (note that $a \in S_P$ was chosen with $\mathsf{v}_P(a) =1$). We endow $\N_0 \times S_P^\times$ with a total order compatible with the monoid operation in the following way: $\N_0$ carries the canonical order $\leq$. On $S_P^\times$ take any total order $\leq$ compatible with the group operation, which is possible by~\cite[Corollary 3.4]{Gil84}. Now $\N_0 \times S_P^\times$ is a totally ordered monoid with lexicographic order, i.e.,
\begin{align*}
(m,s) \leq (n,t) \Leftrightarrow (m < n \lor (m = n \land s \leq t)),
\end{align*}
for $m,n \in \N_0$ and $s,t \in S_P^\times$.\\
Now we can write $h_1 = \sum_{i=1}^n k_i X^{(a_i,b_i)}$ and $h_2 = \sum_{j=1}^m l_j X^{(c_j,d_j)}$ with $k_i, l_j \in K \setminus \{0\}$, $(a_i,b_i),(c_j,d_j) \in \N_0 \times S_P^\times$, and $(a_1,b_1) < (a_2,b_2) < \ldots < (a_n,b_n)$, $(c_1,d_1) < (c_2,d_2) < \ldots < (c_m,d_m)$. Also $g = X^{(0,g_1)} + \ldots + X^{(0,g_n)} + X^{(1,g_n)}$ in this notation. \\
It follows that $(1, g_n) = (a_n,b_n) + (c_m,d_m) = (a_n+c_m,b_n+d_m)$, thus $1 = a_n + c_m$. Without loss of generality, we can suppose that $a_n = 0$ and $c_m =1$. Therefore $a_i = 0$ for all $i \in \{1, \ldots,n\}$. Let $j \in \{1, \ldots, m\}$ be minimal such that $c_j = 1$. Then the monomial in $h_1h_2 = g$ with exponent $(a_1,b_1) + (1,d_j) = (1,b_1+d_j)$ has a non-zero coefficient and therefore it holds that $(a_1,b_1) + (c_j,d_j) = (1,b_1+d_j) = (1,g_n) = (1,b_n+d_m) = (a_n,b_n) + (c_m,d_m)$. It follows that $n = 1$ and hence $h_1 = k_1X^{(0,b_1)} \in K[S_P]$ is a unit. This proves that $g \in K[S_P]$ is irreducible. \qedhere[Proof of Claim B]
\end{proof}
\end{proof}

\begin{proof}[Proof of the Theorem]
Let $K$ be the quotient field of $D$ and $G$ be the quotient group of $S$. If $D = K$ or $S = G$, we are done by Lemma \ref{lemma:Kim} and Lemma \ref{lemma:field}. So assume that $D \neq K$ and $S \neq G$. We have an isomorphism $ \mathcal{C}_v(D) \oplus \mathcal{C}_v(S) \to \mathcal{C}_v(D[S])$ via $([I],[J]) \mapsto [I[J]]$. Hence it suffices to prove that every class of the form $[I[J]]$ contains a prime divisor, where $I$ (resp. $J$) is a non-zero (resp. non-empty) $v$-ideal of $D$ (resp. $S$). So let $I$ be a non-zero $v$-ideal of $D$ and let $J$ be a non-empty $v$-ideal of $S$. \\
By Lemma \ref{lemma:approximation}, let $a,b \in K$ and $P \in \mathfrak{X}(D)$ such that $I^{-1} = (a,b)_v$ and $\mathsf{v}_P(\frac{a}{b}) = 1$. Define $p = \frac{a}{b}$. Then $p$ is a prime element of $D_P$. Moreover, let $g_1,\ldots,g_n \in G$ such that $J^{-1} = ((g_1 + S) \cup \ldots \cup (g_n + S))_v$. For each non-negative integer $m$ we can choose $h_1,\ldots,h_m \in J^{-1}$ such that $g_1, \ldots, g_n, h_1, \ldots, h_m$ are pairwise distinct. This works out, because $S$ is torsion-free and hence infinite. Let $h$ be the maximal element in $\{g_1,\ldots,g_n,h_1,\ldots,h_m\}$ for some fixed total order on $G$ and set $M = \{g_1,\ldots,g_n,h_1, \ldots,h_m\} \setminus \{h\}$. Now we set $g = X^h + \sum_{a \in M} pX^a \in D_P[G]$. Then $g$ is prime in $D_P[G]$ by \cite[Lemma 5]{Chang}. Since $K[G]$ is the localization of $D_P[G]$ at all non-zero constants and $gD_P[G]$ is a prime ideal not containing non-zero constants, it follows that $g \in K[G]$ is prime. Moreover, we have $gK[G] \cap D[S] = g (p,1)^{-1}[((g_1+S) \cup \ldots \cup (g_n+S) \cup (h_1 + S) \cup \ldots \cup (h_m + S))^{-1}] = g b (a,b)^{-1} [J] = gb I[J]$, where the first equality follows by Lemma \ref{lemma:Chang}. Hence $gK[G] \cap D[S]$ is a prime divisor lying in the class of $I[J]$. By varying $m$, the construction leads to non-associated prime elements $g \in K[G]$ and hence the assertion follows.

\end{proof}

\section{Past proof attempts} \label{section:examples}

As noted in the beginning of this work, Chang \cite{Chang} states the fact that in a Krull monoid algebra $D[S]$ every divisor class contains a prime divisor. Unfortunately, due to an error in Lemma 7 of his work (Krull monoids in general do not satisfy the approximation property, see \cite[Theorem 26.4]{HK98}), the proof of his main result collapses. Another error that occurs in Chang's argumentation is due to a mistake in \cite[Corollary 5]{Kim}, where Kim asserts that fractional $v$-ideals of Krull monoids $S$ such that $D[S]$ is a Krull domain are always $v$-generated by two elements. A counterexample to this assertion can be found in \cite[Beispiel 3.8]{Skula}. In the following, we give an explicit counterexample to Chang's Lemma 7.

\begin{counterexample}
This example is based on zero-sum theory. For an introduction to this topic, see \cite[Chapter 5]{GHK}. Let $G_0 = \{-2g,-g,g,2g\} \subseteq G$, where $G$ is an abelian group and $g \in G$ is an element of infinite order. The monoid $\mathcal{B}(G_0)$ of zero-sum sequences over $G_0$ is a finitely generated reduced torsion-free Krull monoid that is not factorial. Its height-one spectrum is $\mathfrak{X}(\mathcal{B}(G_0)) = \{\mathfrak{p}_i \mid i \in G_0\}$, where $\mathfrak{p}_i = \{ S \in \mathcal{B}(G_0) \mid i \in \text{supp}(S) \}$. Now set $\alpha_1 = (-2g) \cdot (-2g) \cdot 2g \cdot 2g \cdot (-g) \cdot (-g) \cdot g \cdot g \in \mathcal{B}(G_0)$ and $\alpha_2 = \alpha_1 \cdot \alpha_1 \in \mathcal{B}(G_0)$. Thus, whenever we take $a \in \mathcal{B}(G_0)$, the element $\alpha_2 + a - \alpha_1$ has $\mathfrak{p}_i$-adic valuation at least $2$ for every $i \in G_0$, which is contradiction to Lemma 7 in \cite{Chang}.
\end{counterexample}

\providecommand{\bysame}{\leavevmode\hbox to3em{\hrulefill}\thinspace}
\providecommand{\MR}{\relax\ifhmode\unskip\space\fi MR }
\providecommand{\MRhref}[2]{%
  \href{http://www.ams.org/mathscinet-getitem?mr=#1}{#2}
}
\providecommand{\href}[2]{#2}

\end{document}